\documentclass[11pt]{article}
\usepackage[latin1]{inputenc}
\usepackage{epsfig}
\usepackage{color}
\usepackage[british,english]{babel}
\usepackage{amsthm}
\usepackage{amsmath}
\usepackage{amsfonts}
\usepackage{amssymb}
\usepackage{graphicx}
\newtheorem{corollary}{Corollary}[section]
\newtheorem{defi}[corollary]{Definition}

\newtheorem{lemma}[corollary]{Lemma}
\newtheorem{prp}[corollary]{Proposition}
\newtheorem{remark}[corollary]{Remark}
\newtheorem{thm}[corollary]{Theorem}
\newfont{\sBlackboard}{msbm10 scaled 900}

\newcommand{\mylabel}[1]{\label{#1}
            \ifx\undefined\stillediting
            \else \fbox{$#1$}\fi }
\newcommand{\BE}{\begin{equation}}

\newcommand{\EEQ}{\end{equation}}
\newcommand{\rfb}[1]{\mbox{\rm
   (\ref{#1})}\ifx\undefined\stillediting\else:\fbox{$#1$}\fi}

\newfont{\Blackboard}{msbm10 scaled 1200}

\newfont{\roma}{cmr10 scaled 1200}

\def\CC{\rm \hbox{C\kern-.56em\raise.4ex
         \hbox{$\scriptscriptstyle |$}\kern+0.5 em }}

\newcommand{\mm}    {{\hbox{\hskip 0.5pt}}}

\newcommand{\bluff} {{\hbox{\raise 15pt \hbox{\mm}}}}
scaled\magstep2

\makeatletter
\def\section{\@startsection {section}{1}{\z@}{-3.5ex plus -1ex minus
    -.2ex}{2.3ex plus .2ex}{\large\bf}}
\makeatother
\def\be{\begin{equation}}
\def\ee{\end{equation}}

\title{\textbf { Weak controllability of second order evolution systems and applications}}
\author{Rachid Attia\footnote{Département de
Mathématiques (ACEDP: 05/UR/15-01), Faculté des Sciences de
Monastir, Université de Monastir, 5019 Monastir, Tunisie, e-mail:
rachid.fsm@gmail.com,} \; and \;   Akram Ben
Aissa\footnote{Département de Mathématiques (ACEDP: 05/UR/15-01),
Faculté des Sciences de Monastir, Université de Monastir, 5019
Monastir, Tunisie, e-mail: akram.benaissa@fsm.rnu.tn}}

\date{}
\begin{document}
\maketitle \noindent {\bf Abstract.} {\small Controllability and
observability are important properties of a distributed paramater
systems.The equivalence between the notion of exact observability
and exact controllability  holds in general.
  In this work, we define a new notion of controllability say weak which is related
  to some weak observability inequality and we give the equivalence between. }

\textbf{2010 Mathematics Subject Classification
,}\;93B07,\,93B05,\,93C20,\,35A15.
\section{Introduction}
Problems of control and observations of waves arises in many
different context and for various models. Hence controllability
refers to the possibility of driving the system under consideration
to prescribed final state at a given final time using a control
function. This question is very interesting when the control
function doesn't act everywhere but is rather located in some part
of the domain or in its boundary through suitable actuators.\\
On the other hand, observability refers to the possibility of
measuring the whole energy of the solutions of the free trajectories
(i.e., without control) through partial measurements. It turns out
that these two properties are equivalent and dual one from another.
This is the basis of the so-called Hilbert Uniqueness method
\cite{Lions}.\\

 \vskip 0.2cm Our starting point is the following.
Let $\Omega\subset \mathbb{R}^n,\,n\geq 2$, be an open bounded
domain with a sufficiently smooth boundary
$\partial\Omega=\overline{\Gamma_0}\cup\overline{\Gamma_1}$,such
that  $\Gamma_0,\,\Gamma_1$ are disjoints parts of the boundary
relatively
open in $\partial\Omega,\,\text{int}(\Gamma_0)\neq\emptyset$.\\
 We consider  the following homogenous wave equation
\begin{equation}\label{wee1}
\frac{\partial^2\phi}{\partial t^2}-\Delta\phi=0,\quad\Omega\times
(0,+\infty),
\end{equation}
\begin{equation}\label{wee2}
\phi=0,\quad\partial\Omega\times (0,+\infty),
\end{equation}
\begin{equation}\label{wee3}
\phi(x,0)=\phi^0(x),\,\frac{\partial{\phi}}{\partial
t}(x,0)=\phi^1(x),
\end{equation}
then, using theorem of Hille-Yoshida, one can easily check that
problem (\ref{wee1})-(\ref{wee3}) is well-posed, i.e., for all
$(\phi^0,\phi^1)\in H^1_0(\Omega)\times L^2(\Omega)$, equations
(\ref{wee1})-(\ref{wee2}) admits unique solution
\begin{equation}\label{reee}
\phi\in C([0,+\infty); H^1_0(\Omega))\cap
C^1([0,+\infty);L^2(\Omega)).
\end{equation}

It's well known that the problem of controllability, that's., there
exists a constant $C_0>0$ such that for all $(z^0,z^1)\in
L^2(\Omega)\times H^{-1}(\Omega)$ there exist a control $g\in
L^2([0,T],L^2(\Gamma))$ such that
\begin{equation}\label{css}
\|g\|_{L^2([0,T],L^2(\Gamma_0))}\leq
C_0(\|z^0\|_{L^2(\Omega)}+\|z^1\|_{H^{-1}(\Omega)})
\end{equation}
such that the solution of
\begin{equation}\label{weee1}
\frac{\partial^2z}{\partial t^2}-\Delta z=0,\quad\Omega\times
(0,+\infty),
\end{equation}
\begin{equation}\label{weee2}
z=g,\quad\Gamma_0\times (0,+\infty),
\end{equation}
\begin{equation}\label{weee3}
z=0,\quad\Gamma_1\times (0,+\infty),
\end{equation}
\begin{equation}\label{weee4}
z(x,0)=z^0(x),\;\frac{\partial z}{\partial t}(x,0)=z^1(x),\quad
\Omega,
\end{equation} satisfy
\begin{equation}\label{weee5}
z(x,t)=0,\quad\forall t\geq T.
\end{equation}
is equivalent to the following observability inequality
\begin{equation}\label{eoi}
\int_0^T\int_{\Gamma_0}\left|\frac{\partial\phi}{\partial\nu}\right|^2d\Gamma_0dt\geq
C \|(\phi^0,\phi^1)\|^2_{H^1_0(\Omega)\times (L^2(\Omega)}.
\end{equation}Due to \cite{BLR},\cite{BU} this last inequality is
equivalent to some geometric  conditions\footnote{each ray which
propagates in $\Omega$ and is reflected on $\Gamma_0$ according to
the laws of geometric optics has to meet  $\Gamma_0$ in time less
than $T$.}(CGC)\cite{BLR}  satisfied by the part of the boundary
$\Gamma_0$ and the
time of control $T>0$.\\
\vskip 0.3cm
 As a first example of this paper,  we consider $\Omega=(0,1)\times (0,1)$ and we prove that the solution of the homogenous system
(\ref{wee1})-(\ref{wee3}) satisfy  the following weakly
observability inequality. For the proof, see appendix.
\begin{prp}\label{wec}\cite{AA}
Let $\Gamma_0=\{(0,x_2);\,x_2\in(0,1)\}=\{0\}\times (0,1)$.
There exist $T_0>0,\,C_{T_0}>0$ such that for all $T>T_0$ and for
all $(\phi_0,\phi_1)\in H^1_0(\Omega)\times L^2(\Omega)$ we have
\begin{equation}\label{woi}
\|(\phi_0,\phi_1)\|^2_{L^2(\Omega)\times H^{-1}(\Omega)}\leq
C_{T_0}\int_0^T\int_{\Gamma_0}|\partial_\nu\phi(x,t)|^2d\Gamma_0(x)dt.
\end{equation}
 \end{prp}

\newpage
 We prove that.
\begin{thm}\label{rp}
System (\ref{wee1})-(\ref{wee3}) is weakly observable in time $T>0$,
that's (\ref{woi}) holds true if and only if
 there exists a control $g\in L^2(0,T;L^2(\Gamma_0))$ such that
\begin{equation}\label{cost}
\|g\|_{L^2(0,T;L^2(\Gamma_0))}\leq
C_0(\|\phi^0\|_{H^1_0(\Omega)}+\|\phi^1\|_{L^2(\Omega)}),
\end{equation}
 and that the solution of
\begin{eqnarray}
\frac{\partial^2z}{\partial t^2}(x,t)-\Delta z(x,t)=0,&&
(x,t)\in\Omega\times (0,+\infty),\label{w1}
\\
z(x,t)=g,&& (x,t)\in \Gamma_0\times (0,+\infty),\label{w2}
\\
z(x,t)=0,&&(x,t)\in\partial\Omega\backslash\Gamma_0\times
(0,+\infty),\label{w3}
\\
z(x,0)=z^0(x),&& x\in\Omega.\label{w4}
\\
\frac{\partial z}{\partial t}(x,0)=z^1(x),&& x\in\Omega.\label{w5}
\end{eqnarray}
 satisfy
\begin{equation}\label{nul}
z(x,t)=0,\quad\forall t\geq T.
\end{equation}
\end{thm}
\begin{defi}
With $g$ and $z$ as above satisfy respectively (\ref{cost}) and
(\ref{nul}), we said that the system (\ref{weee1})-(\ref{weee4}) is
weakly controllable in time $T>0$.
\end{defi}
The outline of this paper is as follows. In the second section we
give some background on HUM method needed here, section $3$ contains
the proof of the main result and abstract framework. The last
section is devoted to some applications.
\section{Survey on
HUM method} Before starting the proof of Theorem \ref{rp}, we recall
the HUM method in the classical case. For more details, see
\cite{Lions}.\\
In fact, J. L. Lions gave a systematic method to reduce the study of
exact controllability problem of system (\ref{weee1})-(\ref{weee4})
to obtain some inequality, say observability inequality or inverse
inequality of the adjoint problem (\ref{wee1})-(\ref{wee3}). He
called this method Hilbert Uniqueness method which can be found in
\cite{Lions} and we briefly describe below.\\
 Let $(\phi^0,\phi^1)\in H^1_0(\Omega)\times
L^2(\Omega)$, and we solve the problem (\ref{wee1})-(\ref{wee3})
then, (\ref{wee1})-(\ref{wee3}) admits a unique solution
$$\phi\in C(0,T;H^1_0(\Omega))\cap C^1(0,T;L^2(\Omega)).$$
In addition, we have the following regularity result:
$\frac{\partial\phi}{\partial\nu}\in L^2(0,T;L^2(\Gamma_0))$ and
there exist a constant
 $C_0>0$ such that
 $$\forall(\phi^0,\phi^1)\in\mathcal{D}(\Omega)\times\mathcal{D}(\Omega),\quad\left\|\frac{\partial\phi}{\partial\nu}\right\|^2_{
 L^2(0,T;L^2(\Gamma_0))}\leq
 C_0(\|\phi^0\|^2_{H^1_0}+\|\phi^1\|^2_{L^2}).$$
 This inequality express that the application
 $(\phi^0,\phi^1)\longrightarrow\frac{\partial\phi}{\partial\nu}$
 extends to a continuous linear application from $H^1_0(\Omega)\times L^2(\Omega)$
 to $ L^2(0,T;L^2(\Gamma_0))$. Now, we consider the following
 backward problem associated to (\ref{wee1})-(\ref{wee3})
 \begin{equation}\label{be1}
\frac{\partial^2\psi}{\partial t^2}-\Delta \psi=0,\quad\Omega\times
(0,+\infty),
\end{equation}
\begin{equation}\label{be2}
\psi=-\frac{\partial\phi}{\partial\nu},\quad\Gamma_0\times
(0,+\infty),
\end{equation}
\begin{equation}\label{be3}
\psi=0,\quad\Gamma_1\times (0,+\infty),
\end{equation}
\begin{equation}\label{be4}
\psi(x,T)=\frac{\partial \psi}{\partial t}(x,T)=0,\quad \Omega,
\end{equation}
then, (\ref{be1})-(\ref{be4}) admits unique solution
$$\psi\in C(0,T;L^2(\Omega))\cap C^1(0,T;H^{-1}(\Omega)),$$
thus, $\psi(x,0)$ and $\frac{\partial\psi}{\partial t}(x,0)$ are
well defined in $L^2(\Omega)$ and $H^{-1}(\Omega)$ respectively. In
fact, the density of $\mathcal{D}(\Omega)$ in $H^1_0(\Omega)$ and
$L^2(\Omega)$ allows us to do all the above steps for
$(\phi^0,\phi^1)\in H^1_0(\Omega)\times L^2(\Omega)$. If we can find
$(\phi^0,\phi^1)\in H^1_0(\Omega)\times L^2(\Omega)$ such that
\begin{equation*}
\left\{
\begin{array}{c}
  \psi(x,0)=z^0(x) \\
  \frac{\partial\psi}{\partial t}(x,0)=z^1(x) \\
\end{array}
\right.
\end{equation*}
we resolve the control problem (\ref{weee1})-(\ref{weee4}) with
$g=-\frac{\partial\phi}{\partial\nu}$ and $z=\psi$.
 Hence, for $(\phi^0,\phi^1)\in H^1_0(\Omega)\times L^2(\Omega)$, we define the
following operator
$$\Lambda(\phi^0,\phi^1)=(\frac{\partial\psi}{\partial
t}(0),-\psi(0))\in L^2(\Omega)\times H^{-1}(\Omega).$$ It's easy to
check that $\Lambda$ is a linear continuous operator from
$F=H^1_0(\Omega)\times L^2(\Omega)$ onto  its dual
$F'=L^2(\Omega)\times H^{-1}(\Omega)$.\\
 In fact, let $\phi=(\phi^0,\phi^1)\in F,\quad
 \tilde{\phi}=(\tilde{\phi}^0,\tilde{\phi}^1)\in F$, then
 multiplying (\ref{be1}) by $\tilde{\phi}$ and integrating by parts, we obtain
 $$\left\langle\Lambda(\phi^0,\phi^1),(\tilde{\phi}^0,\tilde{\phi}^1)\right
\rangle_{F',F}=\int_0^T\int_{\Gamma_0}\frac{\partial
\phi}{\partial\nu}\frac{\partial\tilde{\phi}}{\partial\nu}d\Gamma_0dt,$$
in particular
$$\left\langle\Lambda(\phi^0,\phi^1),(\phi^0,\phi^1)\right
\rangle_{F',F}=\int_0^T\int_{\Gamma_0}\left|\frac{\partial
\phi}{\partial\nu}\right|^2d\Gamma_0dt.$$ If we show that the
continuous bilinear form defined on $F\times F$ by
$$((\phi^0,\phi^1),(\tilde{\phi}^0,\tilde{\phi}^1))\longrightarrow\left\langle\Lambda(\phi^0,\phi^1),(\tilde{\phi}^0,\tilde{\phi}^1)\right
\rangle_{F',F}$$ is coercive, then according to the Lax-Milgram
lemma, we have:  for all $(z^0,z^1)\in L^2(\Omega)\times
H^{-1}(\Omega)$, there exist $(\phi^0,\phi^1)\in H^1_0(\Omega)\times
L^2(\Omega)$ such that
$$\Lambda(\phi^0,\phi^1)=(z^1,-z^0).$$ That's
to say that we have solved the problem of exact controllability of
(\ref{weee1})-(\ref{weee4}). The coercivity is equivalent to:  there
exists a constant  $C>0$ such that $\forall\,(\phi^0,\phi^1)\in
H^1_0(\Omega)\times L^2(\Omega)$
\begin{equation}\label{ioe}
\int_0^T\int_{\Gamma_0}\left|\frac{\partial\phi}{\partial\nu}\right|^2d\Gamma_0dt\geq
C \|(\phi^0,\phi^1)\|^2_{H^1_0(\Omega)\times (L^2(\Omega)}.
\end{equation}
Thus, obtaining (\ref{ioe}) is sufficient condition for exact
controllability of (\ref{weee1})-(\ref{weee4}). More precisely, we
show that (\ref{weee1})-(\ref{weee4}) is exactly controllable in
time $T>0$ if and only if (\ref{ioe}) holds. Hence, HUM can reduce
the problem of exact controllability to the obtention of such
inequality (\ref{ioe}) for (\ref{wee1})-(\ref{wee3}). But we cannot
hope to get (\ref{ioe}) without any conditions, in fact several
types of conditions have been considered and in \cite{BLR}, Bardos,
Lebeau and Rauch also Burq and Gérard \cite{BU} gave a  necessary
and sufficient condition in the case of "very regular" geometrical
domain.

\vskip 0.3cm

\section{Proof of main result and abstract setting}

Before starting the proof of our result, we shall make the following
hypothesis:
\begin{equation}\label{holmg}
\Sigma_0\;\text{allows the application of the Holmgren's Uniqueness
Theorem},
\end{equation}where $$\Sigma_0=\Gamma_0\times [0,T].$$

 Let $F$ be the completion of $H^1_0(\Omega)\times L^2(\Omega)$ with respect to the norm
 $$\|(\phi^0,\phi^1)\|^2_F=\int_0^T\int_{\Gamma_0}\left|\frac{\partial
\phi}{\partial\nu}\right|^2d\Gamma_0dt,$$ and such that $$ F\subset
L^2(\Omega)\times H^{-1}(\Omega).$$ In fact, since we assume
(\ref{holmg}) holds true, the functional $\|.\|_F$ is a norm. Let us
recall the Holmgren's Uniqueness Theorem (see Hörmander \cite{HO}).
\begin{thm} If $u\in \mathcal{D}'(\Omega)$ is a solution of a
differential equation
$$P(x,D)u=0$$
with analytic coefficients, then $u=0$ in a neighborhood  of
non-characteristic $C^1$ if this true on one side.
\end{thm}
We recall that a $C^1$ surface $S\subset\mathbb{R}^n$ with normal
$\xi$ at $x$ is said to be non-characteristic at $x$ for $P(x,D)$ if
$$P_m(x,\xi)\neq 0,$$ where $P(x,D)=\sum_{|\alpha|\leq m}a_\alpha(x)D^\alpha$ is a differential operator with principal symbol
$$P_m(x,\xi)=\sum_{|\alpha|\leq m}a_\alpha(x)\xi^\alpha.$$
Since we suppose that (\ref{woi}) holds true, then $\Sigma_0$
doesn't necessarily  satisfy the geometrical control condition, but
satisfying the condition for the Holmgren's Uniqueness Theorem to
apply, then

$$\|(\phi^0,\phi^1)\|^2_F=\int_0^T\int_{\Gamma_0}\left|\frac{\partial\phi}{\partial\nu}\right|^2d\Gamma_0dt$$
 is a norm strictly weaker than the $ H^1_0(\Omega)\times L^2(\Omega)$.

 Define the
operator $\Lambda$ by
$$\Lambda: F\longrightarrow F',\quad
\Lambda(\phi^0,\phi^1)=(-\frac{\partial z}{\partial
t}(x,0),z(x,0)).$$ An important ingredients of the proof of
Proposition \ref{rp} is following technical  result.
\begin{lemma}
$\Lambda$ is an isomorphism of $F$ onto $F'$.
\end{lemma}
\begin{proof}
Clearly $\Lambda$ is a bounded linear operator. The backward problem
associated to (\ref{wee1})-(\ref{wee3}) is
\begin{equation*}
\frac{\partial^2\psi}{\partial t^2}-\Delta \psi=0,\quad\Omega\times
(0,+\infty),
\end{equation*}
\begin{equation*}
\psi=-\frac{\partial\phi}{\partial\nu},\quad\Gamma_0\times
(0,+\infty),
\end{equation*}
\begin{equation*}
\psi=0,\quad\partial\Omega\backslash\Gamma_0\times (0,+\infty),
\end{equation*}
\begin{equation*}
\psi(x,T)=\frac{\partial \psi}{\partial t}(x,T)=0,\quad \Omega.
\end{equation*}
$F$  and $F'$ are in duality by: $\forall\,\phi=(\phi^0,\phi^1)\in
F',\,\tilde{\phi}=(\tilde{\phi}^0,\tilde{\phi}^1)\in F$
$$\left\langle(\phi^0,\phi^1),(\tilde{\phi}^0,\tilde{\phi}^1)\right
\rangle_{F',F}=\int_\Omega
\phi^1\tilde{\phi}^0-\phi^0\tilde{\phi}^1dx,$$ we prove that
$\Lambda$ is coercive. In fact,  applying the Lax-Milgram theorem,
it sufficies to show the existence of a constant $c>0$ such that
$$\left\langle\Lambda(\phi^0,\phi^1),(\phi^0,\phi^1)\right\rangle_{F',F}\geq
c\|(\phi^0,\phi^1)\|^2_F,\quad\forall\,(\phi^0,\phi^1)\in F.$$ In
fact,  multiplying (\ref{be1}) by $\phi$ and integrating by parts,
we obtain
\begin{equation*}
\begin{split}
0&=\int_0^T\int_{\Omega}\phi(\psi''-\Delta\psi)dxdt=\left[\int_{\Omega}(\phi\psi'-\phi'\psi) dx\right]^T_0\\
&+\int_0^T\int_{\Omega}\phi''\psi dxdt-\int_0^T\int_{\Omega}(\Delta\psi)\phi dxdt\\
\end{split}
\end{equation*} and therefore
$$\int_{\Omega}[\phi^0\psi'(0)-\phi^1\psi(0)]dx=\int_0^T\int_{\Gamma_0}\left|\frac{\partial
\phi}{\partial\nu}\right|^2d\Gamma_0dt.$$ Consequently,
\begin{equation*}
\left\langle\Lambda(\phi^0,\phi^1),(\phi^0,\phi^1)\right\rangle_{F',F}=\int_0^T\int_{\Gamma_0}\left|\frac{\partial
\phi}{\partial\nu}\right|^2d\Gamma_0dt\geq
c\|(\phi^0,\phi^1)\|^2_F,\quad\forall\,(\phi^0,\phi^1)\in
H^1_0(\Omega)\times L^2(\Omega)
\end{equation*}
and by density argument, we get
\begin{equation*}
\left\langle\Lambda(\phi^0,\phi^1),(\phi^0,\phi^1)\right\rangle_{F',F}=\int_0^T\int_{\Gamma_0}\left|\frac{\partial
\phi}{\partial\nu}\right|^2d\Gamma_0dt\geq
c\|(\phi^0,\phi^1)\|^2_F,\quad\forall\,(\phi^0,\phi^1)\in F.
\end{equation*}
\end{proof}

\begin{proof}(of Theorem \ref{rp}).
As we have seen in the previous lemma, we have the following
inequality
$$\left\langle\Lambda(\phi^0,\phi^1),(\phi^0,\phi^1)\right\rangle_{F',F}=\int_0^T\int_{\Gamma_0}\left|\frac{\partial
\phi}{\partial\nu}\right|^2d\Gamma_0dt\geq
c\|(\phi^0,\phi^1)\|^2_F,\quad\forall\,(\phi^0,\phi^1)\in F,$$ on
the other hand
\begin{equation*}
\begin{split}
\|(\phi^0,\phi^1)\|^2_F&\leq
C\int_0^T\int_{\Gamma_0}\left|\frac{\partial
\phi}{\partial\nu}\right|^2d\Gamma_0dt\\
&=C\left\langle\Lambda(\phi^0,\phi^1),(\phi^0,\phi^1)\right\rangle_{F',F}\\
&=C\left\langle(\psi'(0),-\psi(0)),(\phi^0,\phi^1)\right\rangle_{F',F}\\
&\leq C\|(z^1,z^0)\|_{F'}\|(\phi^0,\phi^1)\|_F\\
\end{split}
\end{equation*}
hence
\begin{equation}\label{azert}
\|(\phi^0,\phi^1)\|_F\leq C\|(z^1,z^0)\|_{F'}.
\end{equation}
Let $g=-\frac{\partial\phi}{\partial\nu}$, then
\begin{equation*}
\begin{split}
\left\|g\right\|^2_{L^2(0,T;L^2(\Gamma_0))}&=\int_0^T\int_{\Gamma_0}\left|\frac{\partial
\phi}{\partial\nu}\right|^2d\Gamma_0dt\\
&=\left\langle(z^1,-z^0),(\phi^0,\phi^1)\right\rangle_{F',F}\\
&\leq \|(z^1,z^0)\|_{F'}\|(\phi^0,\phi^1)\|_F\\
\end{split}
\end{equation*}
by (\ref{azert}), we get
$$\left\|g\right\|_{L^2(0,T;L^2(\Gamma_0))}\leq C
\|(z^0,z^1)\|_{F'}\leq C \|(z^0,z^1)\|_{H^1_0(\Omega)\times
L^2(\Omega)},$$ \noindent
and $g$ derives the system (\ref{weee1})-(\ref{weee4}) to zero.\\
Conversely, suppose that there exists a control $g\in
L^2([0,T],;L^2(\Gamma_0))$ satisfying (\ref{cost}) and that the
solution of (\ref{weee1})-(\ref{weee4}) verify (\ref{nul}),
therefore, by the previous lemma we have in particular that
$\Lambda^{-1}$ is continuous , in particular it verifies
$$\|\Lambda(\phi^0,\phi^1)\|_{F'}\geq C\|(\phi^0,\phi^1)\|_F,\quad
\text{for some }\,C>0,$$ and then by continuous imbedding of  $F$ in
$L^2(\Omega)\times H^{-1}(\Omega)$,  we get (\ref{woi}).
\end{proof}

\vskip 0.3cm

\subsection{Variational approach}
In this section we will see how the weak controllability property of
system (\ref{weee1})-(\ref{weee4}) is a consequence of (\ref{woi})
by a minimization method which yields the control of minimal
$L^2(0,T;L^2(\Gamma))$-norm. Spaces $L^2(\Omega)\times
H^{-1}(\Omega)$ and $H^1_0(\Omega)\times L^2(\Omega)$ are in duality
by
$$\left\langle(\phi^0,\phi^1),(z^0,z^1)\right\rangle=\langle\phi^0,z^1\rangle_{H^{-1},H^1_0}-\langle\phi^1,z^0\rangle_{L^2},$$ for all $(\phi^0,\phi^1)\in
L^2(\Omega)\times
H^{-1}(\Omega) $. \\
\\
Let us consider the functional $\mathcal{J}:L^2(\Omega)\times
H^{-1}(\Omega)\longrightarrow \mathbb{R}$ defined by
\begin{equation}\label{min}
\mathcal{J}(\phi^0,\phi^1)=\frac{1}{2}\int_0^T\int_{\Gamma_0}\left|\frac{\partial
\phi}{\partial\nu}\right|^2d\Gamma_0dt+\left\langle(\phi^0,\phi^1),(z^0,z^1)\right\rangle.
\end{equation} where $\phi$  is the solution of the
homogenous system (\ref{wee1})-(\ref{wee3}) with initial data
$(\phi^0,\phi^1)\in L^2(\Omega)\times
H^{-1}(\Omega)$.\\
\\
Thus we have:
\begin{thm}
Let $(\phi^0,\phi^1)\in L^2(\Omega)\times H^{-1}(\Omega)$ and
suppose that $(\tilde{\phi}^0,\tilde{\phi}^1)\in L^2(\Omega)\times
H^{-1}(\Omega)$ is a minimizer of $\mathcal{J}$. If $\tilde{\phi}$
is the corresponding solution of (\ref{wee1})-(\ref{wee3}) with
initial data $(\tilde{\phi}^0,\tilde{\phi}^1)$ then
\begin{equation}\label{vert}
g=-\frac{\partial\tilde{\phi}}{\partial\nu}|_{\Gamma_0}
\end{equation} is a control which leads $(z^0,z^1)$ to zero in time
$T$.
\end{thm}
Let us give sufficient conditions ensuring the existence of a
minimizer of $\mathcal{J}$. For that we recall the following
fundamental result in the calculus of variations which is a
consequence of the so called Direct method for the calculus of
variations. For a proof, see \cite{ET}.
\begin{prp}\label{tp}
Let $H$ be a reflexive Banach space, $K$ a closed convex subset of
$H$ and $\phi:K\longrightarrow\mathbb{R}$ is a function with the
following properties:
\begin{enumerate}
\item $\phi$ is convex
\item $\phi$ is lower semi-continuous
\item If $K$ is unbounded then $\phi$ is coercive, i.e.,
$$\displaystyle{\lim _{\|x\|\rightarrow
+\infty}}\, \phi(x)=+\infty.$$
\end{enumerate}
Then $\phi$ attains its minimum in $K$, i. e, there exists $x_0\in
K$ such that
\begin{equation}
\phi(x_0)=\displaystyle{\min _{x\in K}}\, \phi(x).
\end{equation}
\end{prp}
  As a consequence, we get the following.
\begin{thm}
Let $(\phi^0,\phi^1)\in H^1_0(\Omega)\times L^2(\Omega)$ and suppose
that (\ref{woi})  holds, that's system (\ref{wee1})-(\ref{wee3}) is
weakly observable in time $T$. Then the functional $\mathcal{J}$
defined by (\ref{min}) has a unique minimizer
$(\tilde{\phi}^0,\tilde{\phi}^1)\in L^2(\Omega)(\Omega)\times
H^{-1}(\Omega)$.
\end{thm}
\vskip 0.5cm
\begin{proof} Denote by
$$E_{-1}=L^2(\Omega)\times
H^{-1}(\Omega),\qquad E_0=H^1_0(\Omega)\times L^2(\Omega).$$
 Continuity and convexity are easy to prove.
The existence of minimum of $\mathcal{J}$ is ensured is also
coercive, that's
$$\displaystyle{\lim _{\left\|(\phi^0,\phi^1)\right\|_{E_{-1}}\longrightarrow+\infty}}\,
\mathcal{J}(\phi^0,\phi^1)=+\infty.$$ In fact, coercivity of
$\mathcal{J}$ follows from (\ref{woi}), indeed
\begin{equation*}
\begin{split}
&\mathcal{J}(\phi^0,\phi^1)\geq
\frac{1}{2}\left(\int_0^T\int_{\Gamma_0}\left|\frac{\partial
\phi}{\partial\nu}\right|^2d\Gamma_0dt-\left\|(z^0,z^1)\right\|_{E_{-1}}\left\|(\phi^0,\phi^1)\right\|_{E_0}\right)\\
&\geq\frac{C}{2}\left\|(\phi^0,\phi^1)\right\|^2_{E_{-1}}-\frac{1}{2}\left\|(z^0,z^1)\right\|_{E_{-1} }\left\|(\phi^0,\phi^1)\right\|_{E_0}.\\
\end{split}
\end{equation*}
Hence, we conclude from the previous theorem that $\mathcal{J}$ has
a minimizer $(\tilde{\phi}^0,\tilde{\phi}^1)\in E_{-1}$.  Next we
prove that $\mathcal{J}$ is strictly  convex and hence the minimizer
is unique. In fact, let $(\phi^0,\phi^1),(\psi^0,\psi^1)\in E_{-1}$
and $\lambda\in [0,1]$. We have
\begin{equation*}
\begin{split}
&\mathcal{J}(\lambda(\phi^0,\phi^1)+(1-\lambda)(\psi^0,\psi^1))=\\
&\lambda\mathcal{J}(\phi^0,\phi^1)+(1-\lambda)\mathcal{J}(\psi^0,\psi^1)-\frac{\lambda(1-\lambda)}{2}\int_0^T\int_{\Gamma_0}\left|\frac{\partial
\phi}{\partial\nu}-\frac{\partial\psi}{\partial\nu}\right|^2d\Gamma_0dt.
\end{split}
\end{equation*}
From (\ref{woi}) we deduce that
$$\int_0^T\int_{\Gamma_0}\left|\frac{\partial
\phi}{\partial\nu}-\frac{\partial\psi}{\partial\nu}\right|^2d\Gamma_0dt\geq
\left\|(\phi^0,\phi^1)-(\psi^0,\psi^1)\right\|_{E_{-1}},$$ and
hence, for any $(\phi^0,\phi^1)\neq (\psi^0,\psi^1)$,
$$\mathcal{J}(\lambda(\phi^0,\phi^1)+(1-\lambda)(\psi^0,\psi^1))<\lambda\mathcal{J}(\phi^0,\phi^1)+(1-\lambda)\mathcal{J}(\psi^0,\psi^1).$$
\end{proof}
\begin{remark}{\rm

 According to the previous theorem and under hypothesis
(\ref{woi}), system (\ref{weee1})-(\ref{weee4}) is controllable, and
this control may be obtained as in (\ref{vert}) from the solution of
the homogenous system (\ref{wee1})-(\ref{wee3}) with initial data
minimizing the functional $\mathcal{J}$. }
\end{remark}
\subsection{Abstract framework}
 Let $H$ be a Hilbert space, and let
$A_1:\mathcal{D}(A_1)\longrightarrow H$ be a self-adjoint, positive,
and boundedly invertible operator. We introduce the scale of Hilbert
spaces $H_\alpha,\;\alpha\in\mathbb{R}$, as follows: for every
$\alpha\geq 0,\;H_\alpha=\mathcal{D}(A^{\alpha}_1)$, with the norm
$\|z\|_{\alpha}=\|A^{\alpha}_1z\|_H$. The space $H_{-\alpha}$ is
defined by duality with respect to the pivot space $H$ as follows:
$H_{-\alpha} = H^*_\alpha $ for $ \alpha > 0$. Equivalently,
$H_{-\alpha}$ is the completion of $H$  with respect to the norm
$\|z\|_{-\alpha}=\|A^{-\alpha}_1z\|_H$. The operator $A_1$ can be
extended (or restricted) to each $H_\alpha$ such that it becomes a
bounded operator
$$A_1: H_\alpha\longrightarrow H_{\alpha-1}\quad
\forall\,\alpha\in\mathbb{R}.$$
 Let $B_1$ be an  unbounded linear operator
from $U$ to $H_{-\frac{1}{2}}$, where $U$ is another Hilbert space.
We identify $U$ with its dual, so that $U =U^*$. The systems we
consider are described by
\begin{equation}\label{seo1}
\ddot{z}(t)+A_1z(t)=B_1 v(t),
\end{equation}
\begin{equation}\label{seo2}
z(0)=z^0,\quad \dot{z}(0)=z^1.
\end{equation}
where $t\in [0,+\infty)$ is the time and  $v\in L^2([0,+\infty);U)$.
The equation (\ref{seo1}) is understood as an equation in
$H_{-\frac{1}{2}}$, i.e., all the terms are in $H_{-\frac{1}{2}}$.\\
Let us now  consider the initial value problem
\begin{equation}\label{fre1}
\ddot{\phi}(t)+A_1\phi(t)=0,
\end{equation}
\begin{equation}\label{fre2}
\phi(0)=z^0,\quad \dot{\phi}(0)=z^1,
\end{equation}
It's well known that (\ref{fre1})-(\ref{fre2}) is well-posed in
$\mathcal{D}(A_1)\times\mathcal{D}(A^{\frac{1}{2}}_1)$ and in
$\mathcal{D}(A^{\frac{1}{2}}_1)\times H$. System
(\ref{seo1})-(\ref{seo2}) is well-posed, in fact
\begin{prp}
Suppose that $v\in L^2([0,T];U)$ and that the solutions  $\phi$ of
(\ref{fre1})-(\ref{fre2}) are such that $B^*_1\phi(.)\in
H^1([0,T];U)$ and there exists a constant $C>0$ such that
$$\|(B^*_1\phi)'(.)\|_{L^2(0,T;U)}\leq
C\|(z^0,z^1)\|_{H_{\frac{1}{2}}\times H},\qquad\forall (z^0,z^1)\in
H_{\frac{1}{2}}\times H.$$ Then the system (\ref{seo1})-(\ref{seo2})
admits a unique solution having the regularity
$$z\in C([0,T];H_{\frac{1}{2}})\cap C^1([0,T];H).$$
\end{prp}

 Next, we give the
definition of $\alpha$-weak observability of
(\ref{seo1})-(\ref{seo2}).
\begin{defi}
The system (\ref{fre1})-(\ref{fre2}) is $\alpha$-weakly observable
if there exist a time $T>0$ and a constant $C_T>0$ such that
\begin{equation}\label{ioff}
\int_0^T\|B^*_1\phi(t)\|^2_Udt\geq
C_T(\left\|z^0\right\|^2_{H_{-\alpha}}+\left\|z^1\right\|^2_{H_{-\alpha-\frac{1}{2}}}),\;\forall\,(z^0,z^1)\in
H_1\times H_{\frac{1}{2}},
\end{equation}
where $\alpha>-\frac{1}{2}$.\\
The system (\ref{fre1})-(\ref{fre2}) is $\alpha$-weakly observable
if it's $\alpha$-weakly  observable in some $T>0$.
\end{defi}
\begin{thm}\label{prpp}
 (\ref{ioff}) holds  if and only if  there exist a control $v\in L^2([0,T];U)$
satisfying
\begin{equation}\label{cout}
\left\|v\right\|^2_{L^2(0,T)}\leq
(\left\|z^0\right\|^2_{H_{\alpha+\frac{1}{2}}}+\left\|z^1\right\|^2_{H_\alpha}),\quad\forall
(z^0,z^1)\in H_1\times H_{\frac{1}{2}},
\end{equation}
such that the solution of (\ref{seo1})-(\ref{seo2}) satisfy
$$z\equiv 0,\quad t\geq T.$$
\end{thm}
The previous theorem allows us to introduce a new notion of
controllability said $\alpha$-weak controllability as follows.
\begin{defi} With $v$ as above in Theorem \ref{prpp}, system
(\ref{seo1})-(\ref{seo2}) is said to be $\alpha$-weakly controllable
in time $T>0$, i.e., $v$ verify (\ref{cout}) and derives the system
(\ref{seo1})-(\ref{seo2}) to zero in time $T>0$.
\end{defi}
\begin{proof}(of Theorem \ref{prpp}).
Let $F$ be the completion of $H_{\frac{1}{2}}\times H$ with respect
to the norm
$$\left\|(z^0,z^1)\right\|^2_F=\int_0^T\|B^*_1\phi(t)\|^2_Udt$$
where $\phi$ is a solution of (\ref{fre1})-(\ref{fre2}) and such
that $$ F\subset H_{-\alpha}\times H_{-\alpha-\frac{1}{2}}.$$ Since
(\ref{ioff}) holds true, then for all $T>0$ there exist $C_T>0$ such
that
$$\int_0^T\|\|B^*_1\phi(t)\|^2_Udt\geq
C_T(\|(z^0,z^1)\|^2_F),\qquad\forall\,(z^0,z^1)\in H_1\times
H_{\frac{1}{2}}.$$ Thus $F$ and $F'$ are in duality by
$$\left\langle(u^0,u^1),(v^0,v^1)\right\rangle_{F',F}=\int_{\Omega}u^0v^1-u^1v^0dx,\quad\forall\,(u^0,u^1)\in
F',\;(v^0,v^1)\in F.$$ and the backward problem  associated to
(\ref{fre1})-(\ref{fre2}) is
\begin{equation}
\ddot{\psi}(t)+A_1\psi(t)=B_1 v(t),
\end{equation}
\begin{equation}
\psi(T)=\dot{\psi}(T)=0,
\end{equation}

 Define the following operator
$\Lambda:F\longrightarrow F'$    by
$$\Lambda(u^0,u^1)=(\frac{\partial\psi}{\partial t}(x,0),-\psi(x,0)).$$ Hence we have
$$\left\langle \Lambda(u^0,u^1),(v^0,v^1)\right\rangle_{F',F}=\int_{\Omega}u^0v^1-u^1v^0dx,\quad\forall\,(u^0,u^1)\in
F,\;(v^0,v^1)\in F.$$ It's easy to check that $\Lambda$ is linear
and continuous operator, in particular
\begin{equation}\label{sou}
\left\langle
\Lambda(u^0,u^1),(u^0,u^1)\right\rangle_{F',F}=\int_0^T\|B^*_1\phi(t)\|^2_Udt\geq
C_T(\|(z^0,z^1)\|^2_F)
\end{equation}
 and hence $\Lambda$ is coercive, which imply by the Lax-Milgram lemma that $\Lambda$ is an isomorphism between $F$ and $F'$.\\
Let $v=B^*_1\phi$. One can easily verify that $v$ is in
$L^2([0,T];U)$(it's simply the closed-loop admissibility hypothesis
or the direct inequality) and for
\begin{equation*}
\left\{
\begin{array}{c}
\psi(0)=z^0\\
\dot{\psi}(0)=z^1
\end{array}\right.
\end{equation*} the solution of (\ref{seo1})-(\ref{seo2})
satisfy
$$z\equiv 0,\qquad\forall t\geq T,$$
and the control $v$ satisfy (\ref{cout}). \\
Conversely, from (\ref{sou}) and the embedding of $F$  in
$H_{-\alpha}\times H_{-\alpha-\frac{1}{2}}$, we easily get
(\ref{ioff}) and therefore system (\ref{fre1})-(\ref{fre2}) is
$\alpha$-weakly observable.
\end{proof}
\begin{corollary}
Suppose that system (\ref{fre1})-(\ref{fre2}) is $\alpha$-weakly
observable (that's equivalent to system (\ref{seo1})-(\ref{seo2}) is
$\alpha$-weakly controllable) and we consider the following
observation
\begin{equation}\label{feed}
y(t)=B^*z(t).
\end{equation}
If for fixed $\delta>0$, the    function defined by
$$H(\lambda)=\lambda B^*_1(\lambda^2I+A_1)^{-1}B_1\in\mathcal{L}(U),\quad\forall\,\lambda\in\mathbb{C}_0$$ is uniformly
bounded on
$\mathbb{C}_\delta=\{s\in\mathbb{C}|\,\mathrm{Re}\,s=\delta>0\}$,
  then system (\ref{fre1})-(\ref{fre2}) is weakly stable. See
  \cite{AT} for more details.

\end{corollary}
\section{Application}

We consider the following initial and boundary problem:
\begin{equation}\label{exd1}
\frac{\partial^2u}{\partial t^2}-\frac{\partial^2u}{\partial
x^2}=v(t)\delta_\xi,\quad(x,t)\in (0,1)\times(0,+\infty),
\end{equation}
\begin{equation}\label{exd2}
u(0,t)=u(1,t)=0,\qquad t\in (0,+\infty),
\end{equation}
\begin{equation}\label{exd3}
u(x,0)=u^0(x),\quad \frac{\partial u}{\partial t}(x,0)=u^1(x),\qquad
x\in (0,1)
\end{equation}
where $\xi\in \mathcal{S}$\footnote{ We  denote by $\mathcal{S}$ the
set of all numbers $\rho\in (0, 1)$  such that $\rho\in\mathcal{Q}$
and if $[0, a_1,...,a_n,...]$  is the expansion of $\rho$ as a
continued fraction, then $(a_n)$ is bounded. Let us notice that
$\mathcal{S}$
 is obviously uncountable and, by classical results on diophantine
approximation (cf. [7], p. 120), its Lebesgue measure is equal to
zero.}. $\delta_\xi$ is the Dirac mass concentrated in the point
$\xi\in (0,1)$. Let
$$H=L^2(0,1),\quad U=\mathbb{R},\quad H_{\frac{1}{2}}=H^1_0(0,1),$$
and
$$A=-\frac{d^2}{dx^2},\quad H_1=H^2(0,1)\cap
H^1_0(0,1),\quad Bk=k\delta_\xi,\;\forall k\in\mathbb{R}.$$ The
homogenous problem associated to (\ref{exd1})-(\ref{exd3}) is
\begin{equation}\label{exd11}
\frac{\partial^2\phi}{\partial t^2}-\frac{\partial^2\phi}{\partial
x^2}=0,\quad (0,1)\times(0,+\infty),
\end{equation}
\begin{equation}\label{exd12}
\phi(0,t)=\phi(1,t)=0,\qquad  (0,+\infty),
\end{equation}
\begin{equation}\label{exd13}
\phi(x,0)=u^0(x),\quad \frac{\partial \phi}{\partial
t}(x,0)=u^1(x),\qquad x\in (0,1)
\end{equation}
{\rm According to Proposition $5.4$ of Ammari-Tucsnak \cite{AT}, see
also \cite{KHT},  the observability inequality concerning the trace
at the point $x=\xi$ of the solutions of (\ref{exd11})-(\ref{exd13})
reads as : For all $T>0$ and $\xi\in\mathcal{S}$,  there exists a
constant $C_{\xi,T}>$ such that
\begin{equation}\label{ih1}
\int_0^T\phi^2(\xi,t)dt\geq C_{\xi,T}
(\|u^0\|^2_{H^{-1}(0,1)}+\|u^1\|^2_{(H^2(0,1)\cap H^1_0(0,1))'}),
\qquad\forall (u^0,u^1)\in H_{\frac{1}{2}}\times H.
\end{equation}

 If we consider  $F$ as  the completion
of $H_{\frac{1}{2}}\times H$ with respect to the norm
$$\|(u^0,u^1)\|^2_F=\int_0^T\phi^2(\xi,t)dt.$$ and such $$F\subset H^{-1}(0,1)\times
(H^2(0,1)\cap H^1_0(0,1))'.$$ If we put
\begin{equation}\label{btt}
u^0(x)=\sum_{n\geq 1}a_n\sin(n\pi x),\;u^1(x)=\sum_{n\geq
1}b_n\sin(n\pi x),
\end{equation}
 with
$$(na_n),\quad(b_n)\subset l^2(\mathbb{R}),$$
then  the dual space of $F$ with respect to the pivot space
$L^2(0,1)$ can be characterized  by
$$ (u^0,u^1)\in F'\,\Longleftrightarrow\,\sum_{n\geq
1}\frac{n^2a^2_n+b^2_n}{\sin^2(n\pi\xi)}<\infty,$$ with $u^0,\,u^1$
as in (\ref{btt}). Therefore, inequality (\ref{ih1}) and  Theorem
\ref{prpp} gives the following corollary.}
\begin{corollary}
For a given time $T>0$ and $\xi\in\mathcal{S}$, there exists a
control $v\in L^2([0,T];\mathbb{R})$ such that
$$\left\|v\right\|^2_{L^2([0,T];\mathbb{R})}\leq
C(\|u^0\|^2_{(H^2(0,1)\cap H^1_0(0,1))}+\|u^1\|^2_{H^1_0(0,1)}),$$
and such that the solution of  (\ref{exd1})-(\ref{exd3}) satisfy
$$u\equiv 0,\quad\forall\,t\geq T.$$
\end{corollary}

\section{Comments}

More generally, the problem of observability refers  to dominate the
solution defined in $\Omega$ of some  Pde's  to the restriction on
portion of the boundary   in appropriate norms. For a large class of
Pde's such estimates are false  without constraints on the Cauchy
data or without geometric hypothesis. F. John \cite{FJ} introduced
estimates of  Hölder and logarithmic dependency type that reads in
our model case for  wave equation with control acting
in a portion of the boundary in the following way:\\
 The logarithmic dependency type is the existence of a constant $C>0$ such that
 for all $(u^0,u^1)\neq 0$,  we have
\begin{equation}\label{lio}
 \|(u^0,u^1)\|_{H^1_0(\Omega)\times L^2(\Omega)}\leq
 \exp\left(C\frac{\|(u^0,u^1)\|_{H^1_0(\Omega)\times L^2(\Omega)}}{\|(u^0,u^1)\|_{ L^2(\Omega)\times
 H^{-1}(\Omega)}}\right)^{\frac{1}{\beta}}\left\|\frac{\partial
 u}{\partial\nu}\right\|_{L^2(]0,T[\times\Gamma_0)}
 \end{equation}
 where $\beta\in (0,1)$.
These estimates can be viewed as  observability inequalities  with
low frequency where  the quantity $$
\frac{\|(u^0,u^1)\|_{H^1_0(\Omega)\times
L^2(\Omega)}}{\|(u^0,u^1)\|_{ L^2(\Omega)\times
 H^{-1}(\Omega)}}$$ is a natural measure of the
frequency of the wave. \\
By the same way, we can study the controllability concept associated
to the weakly observability inequality (\ref{lio}).
\subsection{Appendix}
\begin{proof}(of Proposition \ref{wec})\cite{AA}.
The solution of (\ref{wee1})-(\ref{wee3}) is explicitly given by:
\begin{equation}\label{exp}
\phi(x,t)=\sum_{k\in(\mathbb{N}^*)^2}(\alpha_ke^{i\omega_kt}+\alpha_{-k}e^{-i\omega_kt})\sin(k_1\pi
x_1)\sin(k_2\pi x_2).
\end{equation}
with suitable coefficients $\alpha_k$, and where
$k=(k_1,k_2),\,\omega_k=\pi\sqrt{k^2_1+k^2_2}$. \\
Now by using \cite[Theorem 1]{MM}, we first express the inequality
(\ref{woi}) in terms of the Fourier series. We have
\begin{equation*}
\begin{split}
&\int_0^T\int_{\Gamma_0}|\partial_\nu\phi(x,t)|^2d\Gamma_0(x)dt=\int_0^T\int_0^1|\partial_{x_1}\phi(0,x_2,t)|^2dx_2dt=\\
&\int_0^T\int_0^1\left|\sum_{k\in(\mathbb{N}^*)^2}k_1\pi(\alpha_ke^{i\omega_kt}+\alpha_{-k}e^{-i\omega_kt})\sin(k_2\pi
x_2)\right|^2dx_2dt.\\
\end{split}
\end{equation*}
By using the orthogonality of the family $(\sin(k_2\pi
x_2))_{k\in\mathbb{N}^*}$ in $L^2(0,1)$, we get
$$\int_0^T\int_{\Gamma_0}|\partial_\nu\phi(x,t)|^2d\Gamma_0(x)dt\asymp\sum_{k_2\in\mathbb{N}^*}\int_0^T\left|\sum_{k_1\in\mathbb{N}^*}k_1(\alpha_ke^{i\omega_kt}+\alpha_{-k}e^{-i\omega_kt})\right|^2dt.$$
On the other hand,
$$\sum_{k\in(\mathbb{N}^*)^2}k^2_1(|\alpha_k|^2+|\alpha_{-k}|^2)\geq
\sum_{k\in(\mathbb{N}^*)^2}(|\alpha_k|^2+|\alpha_{-k}|^2)=\|(\phi^0,\phi^1)\|^2_{L^2(\Omega)\times
H^{-1}(\Omega)}.$$Then, we apply again \cite[Theorem 1]{MM}, we take
$$d=N=2,\,\lambda_k=\omega_k=\pi\sqrt{k^2_1+k^2_2},\,\forall
k\in(\mathbb{N}^*)^2,\,p_l=l,\,\forall
l\in\mathbb{N}^*,\,\gamma_1=\gamma_2=\frac{\pi}{2\sqrt{2}}.$$ We
finally get (\ref{woi}) with $T>8$, ie., for $T_0=8$.
\end{proof}

\end{document}